\documentclass{article}
\usepackage{amsmath,amssymb,amsthm}
\usepackage{graphics}

\newtheorem{theorem}{Theorem}[section]
\newtheorem{lemma}{Lemma}[section]
\newtheorem{remark}{Remark}[section]

\newcommand{\ar}{\mathop{\rm AR}\nolimits}

\newcommand{\dg}{\mathop{GT}\nolimits}
\newcommand{\bigo}{\mathop{\rm O}\nolimits}

\newcommand{\so}{\mathop{\rm SO}\nolimits}
\newtheorem{cor}{Corollary}[section]
\newtheorem{prop}{Proposition}[section]
\newtheorem{defn}{Definition}[section]

\title {On rigid Hirzebruch genera}

\author {Oleg R. Musin
\thanks{Research supported in part by NSF
grant DMS-0807640 and NSA grant MSPF-08G-201.}}

\begin{document}
\date{}
\maketitle

\begin{abstract} The classical multiplicative (Hirzebruch) genera of manifolds have the wonderful property which is called rigidity. Rigidity of a genus $h$ means that if a compact connected Lie group $G$ acts on a manifold $X$, then the equivariant genus $h^G(X)$ is independent on $G$, i.e. $h^G(X)=h(X)$.

In this paper we are considering the rigidity problem for stably complex manifolds. In particular, we are proving that a genus is rigid if and only if it is a generalized Todd genus.

\medskip


\medskip


\end {abstract}

\section{Introduction}


Let $U_*$ be the complex bordism ring with coefficients in $R={\Bbb Q}$, ${\Bbb R}$, or ${\Bbb C}$. For a closed smooth stably complex manifold $X$, Hirzebruch \cite{Hir1} defined a multiplicative genus $h(X)$ by a homomorphism $h:U_*\otimes R\to R$.

Recall that according to Milnor and Novikov, two stably complex  manifolds are complex cobordant if and only if they have the same Chern numbers. Therefore, for any  multiplicative genus $h$ there exists a multiplicative sequence of polynomials $\{K_i(c_1,\ldots,c_i)\}$  such that
$
h(X)=K_n(c_1,\ldots,c_n),
$
where the $c_k$ are the Chern classes of $X$ and $n=\dim_{\Bbb C}(X)$ (see \cite{Hir1,Hir2,BN} and Section 2).

Let $U_*^G$ be the ring of complex bordisms of manifolds with actions of a compact Lie group $G$. Then for any homomorphism $h:U_*\otimes R\to R$ we can define the {\it equivariant genus} $h^G$, i.e. a homomorphism
$$
h^G:U_*^G\otimes R\to K(BG)\otimes R
$$
(see details in \cite{krich1}).

A multiplicative genus $h$ is called {\it rigid} if for any connected compact group $G$ the equivariant genus $h^G(X)=h(X)$. For the complex case rigidity means that
$$
h^G:U_*^G\otimes R\to R\subset K(BG)\otimes R,
$$
i.e. $h^G([X,G])$ belongs to the ring of constants. It is well known (see \cite{AH, krich1} or \cite[Sec. 3]{BPR}) that ${\bf S}^1$-rigidity implies $G$-rigidity, i.e. it is sufficient to prove rigidity only for the case $G={\bf S}^1$.

For $G={\bf S}^1$, the universal classifying space $BG$ is ${\Bbb C}${P}$^\infty$, and the ring $K(BG)\otimes R$ is isomorphic to the ring of formal power series  $R[[t]]$. Then for any ${\bf S}^1-$manifold $X$ and a Hirzebruch genus $h$ we have 
$h^{S^1}([X,{\bf S}^1])$ in $R[[t]]$. For instance, in the case of an ${\bf S}^1-$action on an almost complex manifold $X^n$ with isolated fixed points $p_1\ldots,p_m$ with weights $w_{i1},\ldots,w_{in}$, $i=1,\ldots,m$,  $\; h^{S^1}$ can be found explicitly (see \cite{nov,misch,kasp,krich1}):
$$
h^{S^1}([X,{\bf S}^1])=S_h(\{w_{ij}\},t):=\sum\limits_{i=1}^m\prod_{j} {\frac{H(w_{ij}\,t)}{w_{ij}\,t}}, \eqno (1.1) $$
where $H$ is the characteristic series of $h$ (see Section 2). If $h$ is rigid, then from $(1.1)$ it follows that
$$
h(X)=S_h(\{w_{ij}\},t) \; \mbox{ for any } \; t. \eqno (1.2)
$$

Atiyah and Hirzebruch based on the Atiyah-Singer index theorem  proved that  $T_y-$genus is rigid \cite{AH}. (Note that $T_0$ is the famous Todd genus.) Krichever \cite{krich1} gives a proof of rigidity of the $T_{x,y}-$genus using global analytic properties of $S_h(\{w_{ij}\},t)$. It is not hard to see that $(1.2)$ yields the
Atiyah-Hirzebruch formula \cite{AH,krich1}  for an almost complex ${\bf S}^1$-manifold $X$:
$$
T_{x,y}(X)=\sum\limits_{i=1}^m {x^{s_i^+}(-y)^{s_i^-}},
$$
where $s_i^+, s_i^-$ are numbers of positive and negative weights $\{w_{ij}\}$.

Recently, Buchstaber and Ray \cite{BR} (see more details in \cite{BPR}) discovered that a correct formula  for stably complex  ${\bf S}^1$-manifolds is the following: 
$$
T_{x,y}(X)=\sum\limits_{i=1}^m {\varepsilon_i x^{s_i^+}(-y)^{s_i^-}},
$$
where $\varepsilon_i=\pm1$ is the ``sign'' of $p_i$. So Krichever's original formula holds only for almost complex manifolds, where all $\varepsilon_i=1$. Note that his $T_{x,y}-$genus rigidity theorem for  stably complex manifolds is proved in \cite{BPR}. 

Let $h$ be a Hirzebruch genus for oriented manifolds, i.e. $h:\Omega^{\so}_*\otimes R\to R$  is a homomorphism of rings. Since the class of oriented manifolds is greater than the class of unitary manifolds, the family of rigid genera for oriented manifolds  is less than the family of complex rigid genera. Note that $T_{1,1}$ coincides with the famous $L$-genus (or signature) \cite{Hir1}. Atiyah and Hirzebruch proved that the $L$-genus is rigid \cite{AH}. They also proved that the $\hat A-$genus vanishes (i.e. is rigid) for spin manifolds.

Krichever \cite{krich76} extended the Atiyah-Hirzebruch theorem on the $\hat A-$genus to almost complex manifolds. Namely he proved that if for an integer $k>1$ we have $c_1(X)\equiv 0$ (mod\,$k$), then $A_k(X)=0$.

The theory of elliptic genera and elliptic cohomologies which arose in the end of 1980s in the papers of Witten, Ochanine, Landweber, and Stong \cite{witt1, witt2,och, land} was stimulated by Witten's conjecture concerning the rigidity of the  character-index of ``twisted" Dirac operators, or equivalently the rigidity of the equivariant elliptic genera \cite{krich2,land,liu}. Bott, Taubes, Hirzebruch, and Krichever \cite{BT, taubes, Hir3, krich2} have studied Witten's conjecture and its extensions and proved several rigidity theorems (see Liu \cite{liu} for references, historical overview, and  rigidity theorems).  Actually, Krichever \cite{krich2} generalized elliptic genera studied in \cite{BT,Hir3,land,och,taubes,witt1,witt2} and proved that they are rigid for SU-manifolds, i.e. for almost complex manifolds whose first Chern class iz zero.

Note that among genera that were considered above only the family of $T_{x,y}$ genera is rigid for all complex manifolds. For all other rigid cases we need additional assumptions, for instance $c_1(X)=0$ for Krichever's genera. In this paper we show that a multiplicative genus is rigid for all complex manifolds  if and only if it is a generalized Todd genus (Sections 3 and 4).

\section{Hirzebruch genera and algebraic rigidity}

\subsection {The genus of a formal power series}
A sequence of polynomials $K_1(c_1), K_2(c_1,c_2),...$ in variables $c_1, c_2,...$ with coefficients in $R$, where $R={\Bbb Q}$, ${\Bbb R}$, or ${\Bbb C}$, is called {\it multiplicative} if
$$
1 + c_1t + c_2t^2 + c_3t^3 + ... = (1 + a_1t + a_2t^2 + ...)\,(1 + b_1t + b_2t^2 + ...)
$$
implies that
$$
\sum\limits_{j=0}^\infty {K_j(c_1,\ldots,c_j)t^j} = \sum\limits_{i=0}^\infty {K_i(a_1,\ldots,a_i)t^i}\,\sum\limits_{\ell=0}^\infty {K_\ell(b_1,\ldots,b_\ell)t^\ell}.
$$
Let $$H(z)=1+r_1t+r_2t^2+r_3t^3+\ldots$$ be a formal power series in $R[[t]]$ with $r_i=K_i(1,0,\ldots,0)$. The power series $H(t)$ is called the {\it characteristic  (or Hirzebruch's)  power series} of the multiplicative sequence $\{K_j\}$.

Note that a multiplicative sequence $\{K_j\}$ is uniquely defined  by its  characteristic power series $H(t)$ \cite[Lemma 1.2.1]{Hir1}. Moreover, for any formal power series $H(t)=1+r_1t+r_2t^2+\ldots$ there exists a multiplicative sequence $\{K_j\}$ with  Hirzebruch's  power series $H(t)$ \cite[Lemma 1.2.2]{Hir1}.

The genus $h$ of complex manifolds corresponding to $H$ is given by
$$
h(X)=K_n(c_1,\ldots,c_n), \; \; n=\dim_{\Bbb C}(X),
$$
where the $c_k$ are the Chern classes of $X$. Therefore, $H$  defines the homomorphism $h:U_*\otimes R\to R$
 from the complex cobordism ring $U_*\otimes R$ into $R$ (see details in \cite{BN,Hir1,Hir2, nov}).

 Novikov \cite{nov} proved that $t/H(t)$  is equal to $g_h^{-1}(t)$ in $R[[t]]$, where
 $$g_h(t)=\sum\limits_{n=0}^\infty {\frac{h_n}{n+1}\,t^{n+1}}$$
 and
 $$h_n:=[H(t)^{n+1}]_n=h({\Bbb C}\,{\mbox{P}}^n).$$

 \subsection{Linear circle actions on  ${\Bbb C}${P}$^n$} 
 Consider the following action of ${\bf S}^1$ on ${\Bbb C}${P}$^n$:
 $$
 [z_0:z_1:\ldots:z_n] \to  [e^{iw_0\varphi}z_0:e^{iw_1\varphi}z_1:\ldots:e^{iw_n\varphi}z_n]. \eqno (2.1)
 $$
 Denote  ${\Bbb C}${P}$^n$ with the circle action $(2.1)$ by ${\Bbb C}${P}$^n[w_0,\ldots,w_n]$.

  If $w_0,\ldots,w_n$ are distinct integers, then this action has $n+1$ isolated fixed points $p_0,\ldots,p_n$ and weights of the representation of ${\bf S}^1$ in the tangent plane to $p_i$ are $w_{0,i},\ldots,w_{i-1,i},w_{i+1,i},\ldots,w_{n,i}$, where $w_{k,i}=w_k-w_i$.

  Let $H=1+r_1t+r_2t^2+\ldots\; $ be a formal power series. Denote
 $$F_H(t):=\frac{H(t)}{t}=\frac{1}{t}+r_1+r_2t+r_3t^2+\ldots$$
 and
 $$
 S_H(w_0,\ldots,w_n;t):=\sum\limits_{i=0}^n\prod_{j\ne i} {F_H([w_j-w_i]t)}.
 $$

 \begin{prop} Let $H(t)=1+r_1t+\ldots\in R[[t]]$ and let $w_0,\ldots,w_{m}$ be distinct integers. Then $S_H(w_0,\ldots,w_n;t)  \in R[[t]]$ and $[S_H(w_0,\ldots,w_n;t)]_0=h_n$, i.e.
 $$
 S_H(w_0,\ldots,w_n;t)=h_n+s_1t+s_2t^2+\ldots .
 $$
 \end{prop}
\begin{proof} It is known (see \cite{nov,misch,kasp,krich1}) that
 $$
 h^{S^1}({\Bbb C}\,{\mbox{P}}^n[w_0,\ldots,w_n])= S_H(w_0,\ldots,w_n;t).
 $$
 On the other hand,   $h^{S^1}({\Bbb C}\,{\mbox{P}}^n[w_0,\ldots,w_n])=
 h({\Bbb C}\,{\mbox{P}}^n)+s_1t+s_2t^2+\ldots \in R[[t]].$
\end{proof}

 \subsection{Algebraically rigid series}

 \begin{defn}We say that a formal power series $H$ is $n$-algebraically rigid and write $H\in \ar^n$ if for any distinct integers $w_0,\ldots,w_{m}$ and for all $m\in[1,n]$ the formal series $ S_H(w_0,\ldots,w_m;t)$
 is constant in $R[[t]]$, i.e. $$S_H(w_0,\ldots,w_m;t)\in R\subset R[[t]].$$
 If $H\in \ar^n$ for all $n$, then we say that $H$ is strong algebraically rigid and write $H\in \ar^\infty$.
 \end{defn}

 \begin{prop} Rigidity yields strong algebraic rigidity, i.e. for any rigid Hirzebruch genus $h:U_*\otimes R\to R$ its characteristic power series $H\in\ar^\infty$. 
 \end{prop}

 \begin{proof}  Since
 $$
 h^{S^1}({\Bbb C}\,{\mbox{P}}^n[w_0,\ldots,w_n])= S(w_0,\ldots,w_n;t),
 $$
  the rigidity of $h$ implies  $ S_H(w_0,\ldots,w_n;t)=h_n\in R\subset R[[t]]$ for all $n$.
 \end{proof}



 \section{GT  rigid series}

  Let $a\in R$. Consider the Euler characteristic, i.e. the Euler genus $a^nc_n(M^n)$. It is easy to see that for the Euler genus the multiplicative sequence of polynomials is
  $$ K_0=1, \;  K_1=ac_1, \; K_2=a^2c_2, \; K_3=a^3c_3,\; \ldots .$$  Then $$H(t)=E_a(t):=1+at$$ and
  $$h_n=(n+1)\,a^n.$$

 Let
 $$
 H_0(t)=\frac{t}{1-e^{-t}} \, .
 $$
 Then $h$ coincides with the Todd genus and all $h_n=1$ \cite[Lemma 1.7.1]{Hir1}.

 Hirzebruch \cite[Lemma 1.8.1]{Hir1}  considers the $T_y-$genus with
 $$
 H_y(t)=\frac{t(e^{t(1+y)}+y)}{e^{t(1+y)}-1}.
 $$
 Note that $T_0$ is the Todd genus and
 $$
 T_y({\Bbb C}\,{\mbox{P}}^n)=\frac{1-(-y)^{n+1}}{1+y}\,.
 $$

 $T_y-$genus was extended by Krichever \cite{krich1}. He considered the $T_{x,y}-$genus with
 $$
 H_{x,y}(t)=\frac{t(xe^{t(x+y)}+y)}{e^{t(x+y)}-1}.
 $$
 Then
 $$T_{x,y}({\Bbb C}\,{\mbox{P}}^n)=\frac{x^{n+1}-(-y)^{n+1}}{x+y}\,. \eqno (3.1)$$

 Let $D_{a,b}(t):=H_{a+b,a-b}(t)$. Clearly,
 $$D_{a,b}(t)=t(a\coth{(at)}+b).$$
 Note that $D_{a,b}(t)$ is well defined for $a=0$ and $$D_{0,b}(t)=E_b(t).$$

 Here we introduce one more family of genera. Let
  $$G_{a,b}(t):=t(a\cot{(at)}+b).$$
Since $\cot(x)=i\coth(ix)$, we have
$$
G_{a,b}(t)=D_{ia,b}(t). \eqno (3.2)
$$
Therefore, for $R={\Bbb C}$, families of series $D$ and $G$ are the same.

Using $(3.1)$ and $(3.2)$ we get
$$
g_{a,b}({\Bbb C}\,{\mbox{P}}^n)=\frac{(b+ia)^{n+1}-(b-ia)^{n+1}}{2ia}\,,
$$
where $g_{a,b}$ is a genus with the characteristic series $G_{a,b}$.

  \begin{defn}
  We say that a formal power series $H\in R[[t]]$ is the GT (Generalized Todd)  series if there are $a,b\in R$ such that $H(t)=D_{a,b}(t)$ or  $H(t)=G_{a,b}(t)$. (For the case $R={\Bbb C}$ we may assume that $H(t)=D_{a,b}(t)$  only.)
  \end{defn}

  \begin{theorem} If $H\in R[[t]]$ is the $\dg$ series, then $h:U_*\otimes R\to R$ is rigid.
  \end{theorem}
  \begin{proof} Proofs of rigidity of $T_y$ and  $T_{x,y}$ genera in \cite{AH, krich1} are  for the case $R={\Bbb Q}$. 
  We can repeat these proofs (with ``signs'' from \cite{BPR})  practically without changes for the case $R={\Bbb C}$. Therefore, the theorem holds for
  $T_{x,y}, \; x,y\in{\Bbb C},$ genera, i.e. for $D_{a,b}(t), \; a,b\in{\Bbb C},$ series. From this it follows that the theorem holds also for $H=D_{a,b}, \; a,b\in R$, where $R={\Bbb Q}$ or ${\Bbb R}$. Moreover,
  $(3.2)$ implies a proof for the case $H=G_{a,b}$.

  We also can  propose  a proof which is based  on multiplicative generators of  $U_*^{S^1}$  \cite{mus}, where $U_*^{S^1}$ is the bordisms ring of complex manifolds with circle actions.  The idea of this proof is very natural: we just verify the theorem for generators that are given in \cite{mus} explicitly. It is clear that rigidity of  $h$  for generators implies rigidity for all manifolds. Note that \cite{mus} gives a proof of the Atiyah - Hirzebruch theorem for the $T_y-$genus.
  \end{proof}

 \section{A converse theorem on rigid genera}

 In this section we prove the main theorem which implies that a Hirzebruch genus is rigid if and only if its characteristic is the GT  series. In our paper \cite{mus2} we consider the case of a rigid genus $h$ such that  $F_H(\log{z})$ is a rational function in $z$. It is proved that in this case we have $H=H_{x,y}$. Surprisingly, the main line  of a proof in \cite{mus2} can be used also for the general case.

 \begin{lemma} If $H=1+r_1t+r_2t^2+\ldots\in\ar^2$, then
 $$F_H^2(-t)+h_1F_H(t)+F_H'(t)=h_2.$$
 \end{lemma}
 \begin{proof} If $H\in\ar^2$, then for any $a,b,c\in R$ we have
 $$
 f((b-a)s)f((c-a)s)+ f((a-b)s)f((c-b)s)+ f((a-c)s)f((b-c)s)=h_2,
 $$
where $f(t)=F_H(t)$.
Let $t=as,\; \varepsilon=bs,$ and $c=-b$. Then
 $$
 f(\varepsilon-t)f(-\varepsilon-t)+ f(t+\varepsilon)f(2\varepsilon)+ f(t-\varepsilon)f(-2\varepsilon)=h_2.
 $$

 Let us denote by $\bigo(\varepsilon)$ any series $\varphi(t,\varepsilon)=\varepsilon\,\psi(t,\varepsilon)$ with  $\psi(t,\varepsilon)\in R[[t,1/t,\varepsilon]]$.  Clearly,
 $$
 \frac{1}{t+\varepsilon}=\frac{1}{t}\,\frac{1}{1+\varepsilon/t}=\frac{1}{t}-\frac{\varepsilon}{t^2}+\ldots= \frac{1}{t}+\bigo(\varepsilon),
 $$
 and $1/(t-\varepsilon)=1/t+\bigo(\varepsilon)$.
 Therefore, we have
 $$
 f(t+\varepsilon)f(2\varepsilon)+ f(t-\varepsilon)f(-2\varepsilon)
 $$
 $$
 = f(t+\varepsilon)\left(\frac{1}{2\varepsilon}+r_1\right)+ f(t-\varepsilon)\left(\frac{-1}{2\varepsilon}+r_1\right)+\bigo(\varepsilon)
 $$
 $$
 = 2r_1f(t)+\frac{f(t+\varepsilon)-f(t-\varepsilon)}{2\varepsilon} +\bigo(\varepsilon)=2r_1f(t)+f'(t)+\bigo(\varepsilon).
 $$
(Note that $f'(t)$ is well defined. Namely, for   $g(t)=\sum_k {g_kt^{k}}$ we have $g'(t)=\sum_k {kg_kt^{k-1}}$.)
 Since $h_1=2r_1$, we obtain $$f^2(-t)+h_1f(t)+f'(t)+\bigo(\varepsilon)=h_2,$$ and
$$F_H^2(-t)+h_1F_H(t)+F_H'(t)=h_2.$$
 \end{proof}

\begin{remark} In fact, we have the following equality:
$$
 h^{S^1}({\Bbb C}\,{{\rm P}}^2[1,0,0]) = F_H^2(-t)+h_1F_H(t)+F_H'(t)
$$
(see \cite{krich1,mus2}). So Lemma 4.1 shows that rigidity of $h$ for ${\Bbb C}\,{{\rm P}}^2[w_0,w_1,w_2]$ with distinct $w_i$ yields that $h$ is rigid for ${\Bbb C}\,{{\rm P}}^2[w_0,w_1,w_1]$ also.
It is clear that the same holds for all $n$, i.e. the rigidity of $h$ for ${\Bbb C}\,{{\rm P}}^n[w_0,\ldots,w_n]$ with distinct $w_i$ yields the rigidity of  $h$  for ${\Bbb C}\,{{\rm P}}^n[w_0,\ldots,w_n]$ with any $w_i$.
\end{remark}

\begin{theorem} $H\in\ar^2$ if and only if $H$ is the $\dg$  series.
\end{theorem}
\begin{proof}
Theorem 3.1 implies that if $H$ is the GT  series, then $H\in\ar^\infty\subset\ar^2.$ (It is easy to prove this fact directly using elementary algebraic computations.)

Consider $H\in\ar^2$. Then $H\in\ar^1$ and
  $$S_H(1,0;t)=f(-t)+f(t)=h_1, \; \mbox{ where } \;  f(t):=F_H(t).$$
This yields
$$
f(-t)=h_1-f(t). \eqno (4.1)
$$

Lemma 4.1 implies
$$
(f(-t))^2+h_1f(t)+f'(t)=h_2.
$$
Using $(4.1)$, we get
$$
(h_1-f(t))^2+h_1f(t)+f'(t)=h_2.
$$
Then
$$
f'=-f^2+h_1f+h_2-h_1^2.
$$
From this follows that
$$
t=\int {\frac{df}{-f^2+h_1f+h_2-h_1^2}}\,. \eqno (4.2)
$$

Let $$d=h_2-3h_1^2/4=h_2-3r_1^2.$$
Consider two cases: (a) $d\ne0$ and (b) $d=0$

\medskip

\noindent (a) $d\ne0. \; $ Let $q_1=r_1-\sqrt{d}, \; q_2=r_1+\sqrt{d}$. Then
$$
\frac{q_2-q_1}{-f^2+h_1f+h_2-h_1^2}=\frac{1}{f-q_1}-\frac{1}{f-q_2}\,. 
$$
It follows from $(4.2)$ that
$$
(q_2-q_1)\,t=\log{\left(\frac{f-q_1}{f-q_2}\right)}+c.
$$
Since $f(t)-1/t\in R[[t]]$ we have $c=0$.    Thus,
$$
H(t)=tf(t)=\frac{t(q_2e^{(q_2-q_1)t}-q_1)}{e^{(q_2-q_1)t}-1}=H_{q_2,-q_1}(t)=D_{\sqrt{d},r_1}(t). \eqno (4.3)
$$

For the case $R={\Bbb C}$, and for the cases $R={\Bbb Q}$ or ${\Bbb R}$ and $d>0$ the equality (4.3) shows that $H$ is the GT  series.

In the cases $R={\Bbb Q}$ or ${\Bbb R}$ and $d<0$ we have
$$
H(t)=H_{q_2,-q_1}(t)=D_{ia,r_1}(t)=G_{a,r_1}(t), \; \mbox{ where } \; a=\sqrt{-d}.
$$

\medskip

\noindent (b) $d=0. \; $ We have
$$
\frac{1}{-f^2+h_1f+h_2-h_1^2}=\frac{-1}{(f-r_1)^2}\,.
$$
It follows from $(4.2)$  that
$$
H(t)=tf(t)=1+r_1t=E_{r_1}(t).
$$
\end{proof}

\begin{cor} Let $R={\Bbb Q}$, ${\Bbb R}$, or ${\Bbb C}$. A Hirzebruch genus $h:U_*\otimes R\to R$ is rigid if and only if $H$ is the $\dg$ series.
\end{cor}

\begin{cor} If $H$ is  2-algebraically rigid, then $H$ is strong algebraically rigid. In other words,
$\ar^2=\ar^3=\ldots=\ar^\infty$.
\end{cor}


Consider the case of oriented manifolds. In this case the characteristic series $H$ has to be even, i.e. $H(-t)=H(t)$ \cite{Hir1}. If we apply Theorem 4.1 for even series, then we get  the following theorem.

\begin{theorem} Let $R={\Bbb Q}$, ${\Bbb R}$, or ${\Bbb C}$. A Hirzebruch genus
$h:\Omega^{\so}_*\otimes R\to R$ is rigid if and only if $H(t)=at\coth{(at)}$ or $H(t)=at\cot{(at)}$, where $a\in R$.
\end{theorem}


\medskip

\medskip

\medskip

\noindent {\bf\Large Acknowledgments}

\medskip

\medskip


The author thanks V. M. Buchstaber, I. M. Krichever, S. P. Novikov, and T. E. Panov  for useful discussions and comments.

\medskip

\medskip

\medskip

\medskip

 O. R. Musin, Department of Mathematics, University of Texas at Brownsville, 80 Fort Brown, Brownsville, TX, 78520.

 {\it E-mail address:} oleg.musin@utb.edu


\medskip

\medskip


\begin{thebibliography}{99}

\bibitem{AH}
M. F. Atiyah and  F. Hirzebruch, Spin manifolds and group actions, in Essays in Topology and Related Subjects, Springer-Verlag, Berlin, 1970, pp. 18-28.

\bibitem{BT}
R. Bott and C. Taubes, On the rigidity theorems of Witten, {\it J. Amer. Math. Soc.} {\bf 2} (1989) 137-186.


\bibitem{BN}
V. M. Buchstaber and S. P. Novikov, Formal groups, power systems and Adams operators, {\it Math. USSR-Sb.} {\bf 13} (1971), 80-116.


\bibitem{BPR}
V. M. Buchstaber, T. E. Panov, and N. Ray. Toric genera, to appear in {\it Internat. Math. Research Notes,} 2010; arXiv:0908.3298.

\bibitem{BR}
V. M. Buchstaber and N. Ray,  The universal equivariant genus and Krichever's formula, {\it Russian Math. Surveys}, 62:1 (2007), 178 - 180.


\bibitem{Hir1}

F. Hirzebruch, Topological Methods in Algebraic Geometry, 3rd edition, Grundlehren der mathematischen Wissenschaften, no. 131, Springer, Berlin-Heidelberg 1966.

\bibitem{Hir3}
F. Hirzebruch, Elliptic genera of level N for complex manifolds, {\it Differential Geometric Methods in Theoretical Physics,} Kluwer, Dordrecht, 1988, 37-63.


\bibitem{Hir2}
F. Hirzebruch, T. Berger, and R. Jung, Manifolds and Modular Forms, Max-Planck-Institut f\"ur Math., Bonn, 1992.

\bibitem{kasp}
G. G. Kasparov, Invariants of classical lens spaces in bordism theory, {\it Izv. Akad. Nauk SSSR Ser. Mat.}, {\bf 33}  (1969), 735-747.

\bibitem{krich1}
I. M. Krichever, Formal groups and the Atiyah - Hirzebruch formula, {\it Izv. Akad. Nauk SSSR Ser. Mat.} {\bf 38} (1974), 1289-1304.

\bibitem{krich76}
I. M. Krichever, Obstructions to the existence of ${\bf S}^1$-actions. Bordisms of ramified covering spaces, {\it Izv. Akad. Nauk SSSR,} {\bf 40} (1976), no. 4, 828-844.

\bibitem{krich2}
I. M. Krichever, Generalized elliptic genera and Baker-Akhiezer functions, {\it Math. Notes} {\bf 47} (1990), no. 2, 132-142.

\bibitem{land}

P. S. Landweber, Elliptic cohomology and modular forms, In Landweber, P. S., Elliptic curves and modular forms in algebraic topology, Lecture Notes in Math., Vol. {\bf 1326}, Springer, Berlin, 1988, 55-68.

\bibitem{liu}
K. Liu, On modular invariance and rigidity theorems, {\it J. Differential Geom.}, {\bf 41} (1995), 343-396.

\bibitem{misch}
A. S. Mischenko, Manifolds with actions and fixed points, {\it Math. Notes}, {\bf 4} (1968), 381-386.

\bibitem{mus}
O. R. Musin, Generators of ${\bf S}^1$-bordisms,  {\it Math. USSR Sb.} {\bf 44} (1983),  325-334

\bibitem{mus2}
O. R. Musin, An inverse theorem on equivariant genera,  {\it Russian Math. Surveys}, {\bf 64} (2009), 753-755. 

\bibitem{nov}
S. P. Novikov, Adams operators and fixed points, {\it Math. USSR-Izv.} {\bf 2} (1968), 1193-1211.

\bibitem{och}
S. Ochanine, Genres elliptiques equivariants, In Landweber, P.S., Elliptic curves and modular forms in algebraic topology, Lecture Notes in Math., Vol. {\bf 1326}, Springer, Berlin, 1988, 107-122.


\bibitem{taubes}
C. Taubes, ${\bf S}^1$-actions and elliptic genera, {\it Comm. Math. Phys.} {\bf 122} (1989) 455-526.


\bibitem{witt1}
E. Witten, Elliptic genera and quantum field theory, {\it Comm. Math. Phys.} {\bf 109} (1987) 525-536.

\bibitem{witt2}
E. Witten, The index of the Dirac operator in loop space, In Landweber, P. S., Elliptic curves and modular forms in algebraic topology, Lecture Notes in Math., Vol. {\bf 1326}, Springer, Berlin, 1988, 161-186.


\end{thebibliography}
\end{document}